\documentclass[12pt]{article}

\usepackage{graphicx,amsmath,amsfonts,amsthm,amssymb,color,hyperref} 

\newtheorem{theorem}{Theorem}[section]
\newtheorem{lemma}[theorem]{Lemma}
\newtheorem{corollary}[theorem]{Corollary}
\newtheorem{proposition}[theorem]{Proposition}

\theoremstyle{definition}
\newtheorem{definition}[theorem]{Definition}

\theoremstyle{remark}
\newtheorem{remark}[theorem]{Remark}

\numberwithin{equation}{section}

\newcommand\bE{{\mathbb E}}
\newcommand\cA{\mathcal{A}}
\newcommand\cF{\mathcal{F}}
\newcommand\bP{{\mathbb P}}
\newcommand\bR{{\mathbb R}}
\newcommand\cP{\mathcal{P}}
\newcommand\cB{\mathcal{B}}

\newcommand\cM{\mathcal{M}}

\newcommand{\dt}{\partial_t}
\newcommand{\dss}{\partial_{ss}}
\newcommand\tP{{\widetilde{\mathbb P}}}
\newcommand{\ds}{\partial_s}
\newcommand{\dy}{\partial_y}
\newcommand{\dyy}{\partial_{yy}}
\newcommand{\dx}{\partial_x}
\newcommand{\dxx}{\partial_{xx}}
\newcommand{\vd}{\mathrm{d}}
\newcommand\cN{\mathcal{N}}

\newcommand\half{\frac{1}{2}}

\newcommand{\beq}{\begin{eqnarray}}
\newcommand{\enq}{\end{eqnarray}}
\newcommand{\ben}{\begin{eqnarray*}}
\newcommand{\enn}{\end{eqnarray*}}

\begin{document}

\title{Geometric martingale Benamou--Brenier transport and geometric Bass martingales}

\author{Julio Backhoff\footnote{University of Vienna, Austria}, Gr\'egoire Loeper\footnote{BNP Paribas, France}, Jan Ob\l\'oj\footnote{Mathematical Institute and St John's College, University of Oxford, UK}}

\maketitle




\abstract{
Bass martingales appear as solutions to the martingale version of the Benamou--Brenier optimal transport formulation. They are continuous martingales on $[0,1]$, with prescribed initial and terminal distributions, which are as close to Brownian motion as possible: their quadratic variation is as close as possible to being linear in the averaged $L^2$ sense. We develop here their geometric counterparts, which track the geometric Brownian motion instead: the quadratic variation of their logarithm is as close as possible to being linear. 
By analogy between the Bachelier and the Black-Scholes models in mathematical finance, the newly obtained \emph{geometric Bass martingales} have the potential to be of more practical importance in a number of applications. 

Our main contribution is to exhibit an explicit bijection between geometric Bass martingales and their arithmetic counterparts. 
This allows us, in particular, to translate fine properties of the latter into the new geometric setting. We obtain an explicit representation for a geometric Bass martingale for given initial and terminal marginals, we characterise it as a solution to an SDE, and we show that geometric Brownian motion is the only process which is both an arithmetic and a geometric Bass martingale. Finally, we deduce a dual formulation for our geometric martingale Benamou--Brenier problem. Our main proof is probabilistic in nature and uses a suitable change of measure, but we also provide PDE arguments relying on duality.\footnote{This research was funded
in part by the Austrian Science Fund (FWF) DOI 10.55776/P36835.}
}


\section{Introduction}

Optimal transport (OT) refers to a wide range of problems all concerned with constructing couplings of measures with optimal properties. It is an area of mathematics with a long history, finding applications across a wide spectrum of fields.

A particularly interesting way of constructing optimal couplings derives from a fluid mechanics perspective, known as the Benamou--Brenier formulation of OT: consider a fluid moving with time, driven by an unknown velocity field,  starting with a given mass distribution $\nu_0$. The problem is to find  the velocity having the smallest average kinetic energy (i.e., least action) such that the  final distribution
of mass is $\nu_1$. Mathematically this amounts to solving the problem
\begin{align}\label{eq:ambb_speed}
\textstyle
\inf_{\substack{X_0\sim \nu_0, X_1\sim \nu_1\\ X_t=X_0+\int_0^t V_u du}} 
\bE \left[\int_0^1|V_t|^2dt\right].
\end{align}
Remarkably, this time-continuous problem is in fact equivalent to a static one of minimising the average squared distance among couplings of $\nu_0$ and $\nu_1$. In this way, the marginal distributions $\{\text{Law}(X_t)\}_{t\in [0,1]}$ of the continuously moving fluid trace the geodesic connecting $\nu_0$ and $\nu_1$ in the space of measures, endowed with the celebrated 2-Wasserstein distance. 

More recently, \cite{BaBeHuKa20} considered the martingale analogue of this problem, which when specialized to the one-dimensional setting becomes: 
\begin{align}
\textstyle
\label{eq:ambb}\tag{A-mBB} 
\inf_{\substack{M_0\sim \nu_0, M_1\sim \nu_1\\ M_t=M_0+\int_0^t \Sigma_u dB_u}} 
\bE \left[\int_0^1(\Sigma_t - \bar\Sigma)^2dt\right],
\end{align}
where the optimisation is taken over filtered probability spaces with a Brownian motion $(B_t)_{t\geq 0}$, and where $B_0$ does not need to be a constant. By classical results, continuous {real-valued} martingales are all time-changes of a Brownian motion and are thus fundamentally characterised by their quadratic variation. Here, $\nu_0,\nu_1$ are given, in convex order and with finite second moments. Problem \eqref{eq:ambb} looks for a particle evolving as a martingale which is the closest to 
a constant $\bar\Sigma>0$ multiple of a Brownian motion. This problem has been dubbed the martingale Benamou--Brenier problem in \cite{BaBeHuKa20}. To stress that the reference process is the \emph{arithmetic} Brownian motion, we refer to it as the \emph{arithmetic}-mBB problem. 
This problem was studied in detail in \cite{BaBeHuKa20,BaBeScTs23}, who show in particular that \eqref{eq:ambb} admits a unique (in distribution) optimiser, known as the \emph{stretched Brownian motion} from $\nu_0$ to $\nu_1$. Under a mild regularity assumption, known as irreducibility and explained in section \ref{sec:prelim} below, the optimiser is further shown to be a \emph{Bass martingale}. This means in particular that $M_t=f(t,B_t)$, for all $t$, where $f(t,\cdot)$ is non-decreasing; see Definition \ref{def:Bassmg} below. Further, they also show that this continuous time problem is equivalent to a static \emph{weak}-OT problem in the sense of \cite{GoRoSaTe14}. We refer to \cite{Ba83} for a classical application of Bass martingales to the Skorokhod embedding problem and to \cite{PaRoSc22} for a much more recent application concerning Kellerer's theorem.

Martingales, and diffusion processes in particular, are a backbone of mathematical finance: they describe the dynamics of risky assets under a pricing measure. Selecting a model involves its calibration, a process ensuring that model prices match the observed market prices. The canonical example of a calibration problem is that of matching European call and put prices. This, via the classical argument of \cite{BrLi78}, is equivalent to matching marginal distributions at some fixed times, the maturities. Finding robust bounds on prices of an exotic option then corresponds to minimising, or maximising, the expectation of a certain path functional over all such martingales. Starting with \cite{Ho98a} this observation underpinned new interplay between Skorokhod embeddings and robust finance, e.g., \cite{BrHoRo01a,CoOb11a}, and subsequently led to the introduction of \emph{martingale optimal transport} in \cite{PHB,GaLaTo} and the ensuing rapid and rich growth of this field. More recently, optimal transport techniques have also been used as means for non-parametric calibration: OT is used as a means to project one's favourite model onto the set of calibrated martingales, i.e., martingales which satisfy a set of given distributional constraints, see \cite{guo2022calibration,guo2022joint}. 
In general, this OT-calibration problem is solved via its dual, numerically optimizing over solutions to a non-linear PDE, which can be challenging. The Bass martingale can be seen as a particular case of the OT-calibration problem, but one which can be reduced to a static problem, and is hence much easier to solve; see section \ref{sec:numerics} below. The main drawback of the resulting solution is that the \emph{arithmetic} Brownian motion is not a desirable model for risky assets. 
Instead, the \emph{geometric} Brownian motion is, and in finance one quantifies the variability of a model using the quadratic variation of the logarithm of the price process. 

Motivated by the above remarks, we consider a \emph{geometric} version of the martingale Benamou--Brenier problem. We suppose $\mu_0,\mu_1$ are supported on $(0,+\infty)$ and study the problem
\begin{align}
\label{eq:gmbb}\tag{G-mBB} 
\mathbf{GmBB}_{\mu_0,\mu_1}&= 
\inf_{\textstyle\substack{S_0\sim \mu_0, S_1\sim \mu_1\\ S_t=S_0+\int_0^t \sigma_u S_u dB_u}} \bE \left[\int_0^1(\sigma_t - \bar\sigma)^2dt\right],
\end{align}
where the optimisation is taken over filtered probability spaces $(\Omega,\cF,(\cF_t)_{t\geq 0},\bP)$ with a Brownian motion $(B_t)_{t\geq 0}$, and $\bar \sigma>0$. We will show that a suitable change of measure argument allows to build a 1-1 relationship between \eqref{eq:gmbb} for $(\mu_0,\mu_1)$ and the \eqref{eq:ambb} for a different set of marginals $(\nu_0,\nu_1)$. This, in particular, yields uniqueness (in distribution) of the solution to \eqref{eq:gmbb}. 
{We will present a probabilistic proof of the equivalence between the two problems, as well as a PDE motivation behind it. Furthermore, the 1-1 relationship between arithmetic and geometric problems means that we can efficiently translate  results from the arithmetic setting into the geometric one. 
For instance \cite{BaBeHuKa20,BaBeScTs23} explore in detail the fine structure of the optimiser in \eqref{eq:ambb}, \cite{CoHe21,AcMaPa23,JoLoOb23} propose a numerical method for it, and \cite{BeJoMaPa21b} establish its stability w.r.t.\ perturbation of the data of the problem. We will show in  section \ref{sec:gBass structure} what this implies for the structure of our geometric optimiser  and we will explain in  section \ref{sec:numerics} what a numerical method for the geometric setting looks like.

The idea of transforming a martingale transport problem into another via a change of measure, but in discrete time, was pioneered by Campi, Laachir and Martini \cite{CaLaMa14}. This is a particular case of a \textit{change of numeraire} argument. We will see how the same idea is fruitful in our continuous time setting as well.\footnote{
Note added in revision: independently from us, Beiglb\"ock,  Pammer and Riess \cite{BPR:24} have recently used similar techniques to tackle \eqref{eq:gmbb} and related weak optimal transport problems.} 
}

\section{Preliminaries}
\label{sec:prelim}
We let $\bR_+=(0,+\infty)$ and denote $\cP_{r,p}(\bR_+)$ for $r<0<p$ the set of probability measures on $\bR_+$ for which $\int_0^\infty |x|^s\mu(dx)<\infty$ for $r\leq s\leq p$. 
The pus-forward operator is denoted with a subscript $\#$, i.e., for a function $G:\bR_+\to \bR_+$ and probability measures $\mu$ and $\nu$, we write 
$$
G_\# \mu=\nu \Leftrightarrow \nu=\mu\circ G^{-1}\Leftrightarrow  \forall \Gamma\in \cB(\bR_+), \ \nu(\Gamma)=\mu(G^{-1}(\Gamma)).
$$
For a $\mu$-integrable $f:\bR_+\to \bR_+$, we consider the $f$-reflected measure 
$$ \textstyle f_\dag \mu = \left(y\to \frac{1}{f(y)}\right)_\# \left(\frac{f(y)}{\int f(x)\mu(dx)}\mu(dy)\right),$$
i.e., we define a new probability measure with density proportional to $f$ w.r.t.\ $\mu$ and then consider its push-forward using $1/f$. We will be mostly interested in the case $f(y)=y$, in which case the measure resulting from the above is denoted $Id_\dag\mu$. Note that $Id_\dag\cdot$ reflects moments and is an involution: $Id_\dag(Id_\dag \mu)=\mu$ for $\mu \in\cP_{-1,2}(\bR_+)$.

It is immediate that \eqref{eq:gmbb} is invariant under multiplicative scaling: the value remains the same and the optimisers are constant multiples of each other (Remark \ref{rk:mean1enough}). It is thus useful to normalise probability measures and for a $\mu\in \cP_{0,1}(\bR_+)$ we let $\tilde \mu:=(x\to x/m)_\#\mu$, where $m=\int x\mu(dx)$. In this way, $\int x\tilde \mu(dx)=1$. We let $$Id_\ddag \mu := Id_\dag \tilde \mu=\widetilde{Id_\dag \mu}.$$


We say that $\eta,\rho\in \cP_{0,1}(\bR_+)$ are in convex order, and write $\eta\preccurlyeq_{cx}\rho$, if $\int f(x)\eta(dx)\leq \int f(x)\rho(dx)$ for all convex functions $f:\bR_+\to \bR_+$. In fact, working on $\bR$, we have $\eta\preccurlyeq_{cx}\rho$ if and only if $U_\eta\leq U_\rho$, pointwise on $\bR$, where $U_\rho(z):=\int_{\mathbb R}|x-z|\rho(dx)$ is known as the potential of $\rho$. 
As potential functions are continuous, the set $\{U_\eta < U_\rho\}$ is open and hence equal to an at most countable union of open maximal intervals. We write $\mathcal I_{[\eta,\rho]}$ for the collection of these intervals and notice that each such interval is contained in $\bR_+$. We say that $\eta,\rho$ are irreducible if $\mathcal I_{[\eta,\rho]}$ contains a single interval $I$. In this case, w.l.o.g., we may and will assume that $\eta(\bR\setminus I)=0$. We refer to \cite{Ob04,BeJu16} for further details on potentials and their applications in Skorokhod embeddings and in martingale optimal transport. 

\begin{lemma}\label{lem:bijection}
Let $\mu_0,\mu_1\in \cP_{-1,1}(\bR_+)$ and $\nu_i:=Id_\dag\mu_i$, $i=0,1$. Then $\nu_0,\nu_1\in \cP_{0,2}(\bR_+)$  and $$\mu_0\preccurlyeq_{cx}\mu_1 \iff \nu_0\preccurlyeq_{cx}\nu_1.$$ 
In case $\mu_0\preccurlyeq_{cx}\mu_1$, we  have 
$I\in \mathcal I_{[\nu_0,\nu_1]} \iff \left\{\frac{1}{x}:x\in I\right\}\in\mathcal I_{[\mu_0,\mu_1]},$ 
and conversely $J\in \mathcal I_{[\mu_0,\mu_1]} \iff \left\{\frac{1}{x}:x\in J\right\}\in\mathcal I_{[\nu_0,\nu_1]}$. 
\end{lemma}
\begin{proof} That $\nu_0,\nu_1\in \cP_{0,2}(\bR_+)$ if $\mu_0,\mu_1\in \cP_{-1,1}(\bR_+)$, is immediate. Let $m_i=\int x\mu_i(dx)$. Now, for $z\in \mathbb R_+$, we have
    $$\textstyle U_{\nu_i}(z) = \int_{\mathbb R_+} |1/x-z|(x/m_i)\mu_i(dx)=
    \frac{z}{m_i}\int_{\mathbb R_+} |1/z-x|\mu_i(dx),$$  so $\frac{m_i}{z}U_{\nu_i}(z )=U_{\mu_i}(1/z)$. 
     As all measures involved are supported in $\mathbb R_+$, we also have $U_{\mu_i}(z)=m_i-z$ and  $U_{\nu_i}(z)=1/m_i-z$ for $z\leq 0$. It follows that $U_{\mu_0}(z)\leq U_{\mu_1}(z)$ for all $z\in \bR$ if and only if $U_{\nu_0}(z)\leq U_{\nu_1}(z)$ for all $z\in \bR$, and in this case $m_1=m_2$.
     
  All open intervals considered are likewise subsets of the positive reals. It follows,  
    for $z>0$, that $U_{\nu_1}(z)>U_{\nu_0}(z)$ if and only if $U_{\mu_1}(1/z)>U_{\mu_0}(1/z)$. This exhibits the desired bijection between $\mathcal I_{[\nu_0,\nu_1]}$ and $\mathcal I_{[\mu_0,\mu_1]}$.
\end{proof}
 We note that the same remains true if $\nu_i=Id_\ddag\mu_i$ with the only difference that 
 \begin{equation}\label{eq:components equivalence}
 \textstyle
 I\in \mathcal I_{[\nu_0,\nu_1]} \iff \left\{\frac{m}{x}:x\in I\right\}\in\mathcal I_{[\mu_0,\mu_1]},    
 \end{equation}
 where $m=\int x\mu_0(dx)=\int x \mu_1(dx)$.
 
We denote by $\gamma_s$ the centred one-dimensional Gaussian distribution with variance $s$. We use $\ast$ to denote the convolution between a function and a measure and $v^\ast$ to denote the convex conjugate (Legendre transform) of a convex function $v$. We write $MC(\eta,\rho)$ for the OT problem of maximal covariance between measures $\eta,\rho$:
$$\textstyle MC(\eta,\rho) = \sup_{\pi\in \Pi(\eta,\rho)}\int xy \pi(dx,dy),$$
where $\Pi(\eta,\rho)$ denotes measures on $\bR^2$ with marginals $\eta$ and $\rho$. This problem is naturally equivalent to the squared Wasserstein distance since
$$\textstyle 2MC(\eta,\rho)= \int x^2 (\eta(dx)+\rho(dx)) - \inf_{\pi\in \Pi(\eta,\rho)}\int (x-y)^2 \pi(dx,dy).$$

\section{Main results}
First, we make a simple observation that allows us to rewrite \eqref{eq:ambb} and \eqref{eq:gmbb} as maximisation problems. We refer the reader to e.g.\ \cite{KaSh91} for concepts and terminology from stochastic analysis. Consider a martingale $M$ admissible in \eqref{eq:ambb}. Note that its quadratic variation is given by $\langle M\rangle_t = \int_0^t\Sigma_u^2du$ and the expectation in \eqref{eq:ambb} is finite if and only if $\langle M\rangle_1$ is integrable, if and only if $M$ is a square integrable martingale and $M_t^2 - \langle M\rangle_t$, $t\in [0,1]$, is also a martingale. It follows that we then have 
$$\textstyle \bE \left[\int_0^1\Sigma_t^2dt\right] = \bE[\langle M\rangle_1]=  \bE[ M_1^2] - \bE[M_0^2] = \int x^2 d\nu_1 - \int x^2 d\nu_0.
$$
Recalling that {$\bar\Sigma >0$}, it follows that the original problem is equivalent to
\begin{equation}\label{eq:ap}\tag{AP}
\mathbf{AP}_{\nu_0,\nu_1} = \sup_{\textstyle\substack{M_0\sim \nu_0, M_1\sim \nu_1\\ M_t=M_0+\int_0^t \Sigma_u dB_u\\ M \text{ martingale}}} \bE \left[\int_0^1 \Sigma_t dt\right],
\end{equation}
in the sense that the two problems share the optimisers and the value in \eqref{eq:ambb} is equal to $\bar\Sigma^2 + \int x^2 d\nu_1 - \int x^2 d\nu_0 - 2 {\bar\Sigma} \mathbf{AP}_{\nu_0,\nu_1}$. Analogously, for a martingale $S$ admissible for \eqref{eq:gmbb} we note that 
$$\log S_t = \log S_0 + \int_0^t \sigma_u dB_u - \frac{1}{2}\int_0^t \sigma^2_u du,\quad t\in [0,1].$$
Localising so that the stochastic integral is a martingale, taking expectations and limits, we see that\footnote{{Assuming, as we will, that $\mu_i \in \cP_{-1,1}(\bR_+)$, it follows that $\int|\log(x)|d\mu_i(x)<\infty$.}}
$$ \textstyle \bE \left[\int_0^1\sigma_t^2dt\right] = 2 \bE[\log(S_0/S_1)] = 2\int \log(x) d\mu_0 - 2\int \log(x) d\mu_1,
$$
and hence \eqref{eq:gmbb} is equivalent to the following problem:
\begin{equation}\tag{GP}\label{eq:gp}
\mathbf{GP}_{\mu_0,\mu_1} = \sup_{\textstyle\substack{S_0\sim \mu_0, S_1\sim \mu_1\\ S_t=S_0+\int_0^t \sigma_uS_u dB_u \\ S \text{ martingale}}} \bE \left[\int_0^1 \sigma_t dt\right],
\end{equation}
where $(B_t)_{t\geq 0}$ is a Brownian motion on a filtered probability space $(\Omega,\cF,(\cF_t)_{t\geq 0},\bP)$. 

We can state our first main result:

\begin{theorem}\label{thm:g-a equiv}
 Let $\mu_0,\mu_1\in \cP_{-1,1}(\bR_+)$ satisfy $\mu_0\preccurlyeq_{cx} \mu_1$. Let $\nu_i=Id_{\ddag}\mu_i$, $i=0,1$. Then \eqref{eq:gp} admits a unique optimiser in distribution and we have
 $$ \mathbf{GP}_{\mu_0,\mu_1} = \mathbf{AP}_{\nu_0,\nu_1}.$$
  \end{theorem}
  In fact the proof of the above result establishes a 1-1 mapping between the optimisers to $\mathbf{GP}_{\mu_0,\mu_1}$  and to $\mathbf{AP}_{\nu_0,\nu_1}$. Using our understanding of the latter we can deduce a detailed description of the former. In the case of multiple irreducible components the full description is more involved and we present it in section \ref{sec:gBass structure}. We state here the result for the important special case of irreducible measures. The existence and structure of the optimiser in \eqref{eq:ap} used below follows from \cite[Thm.~3.1]{BaBeHuKa20} or \cite[Thm.~1.3]{BaBeScTs23}, see also Definition \ref{def:Bassmg} below. 
  
  \begin{theorem}\label{thm:gBass structure irr}
 Let $\mu_0,\mu_1\in \cP_{-1,1}(\bR_+)$ with $\mu_0\preccurlyeq_{cx} \mu_1$ be irreducible. Set $\nu_i=Id_{\ddag}\mu_i$, $i=0,1$, and let $F(1,\cdot)$ be an increasing function,  $\alpha\in\cP(\bR)$, such that $(F(t,B_t),t\in [0,1])$ is an optimiser for $\mathbf{AP}_{\nu_0,\nu_1}$, where $B$ is a Brownian motion with an initial distribution $B_0\sim \alpha$ and $F(t,\cdot)=F*\gamma_{1-t}(\cdot)$. 
   Then the distribution of the optimiser $S$ in $\mathbf{GP}_{\mu_0,\mu_1}$ is characterised by 
\begin{equation}\label{eq:simulateGBass}
    \bE\Big[g\big(\{S_t: t\in [0,1]\}\big)\Big]=\bE\Big[g\big(\{m/F(t,B_t):t\in[0,1]\}\big)\cdot F(1,B_1)\Big],
\end{equation}
 for any measurable functional $g:C([0,1];\bR)\to\bR_+$, where $m:=\int x\mu_0(dx)$.
   \end{theorem}

  Using  that $(F(t,B_t):t\in [0,1])$ is a martingale, for $g:\bR\to \bR_+$ we have
  $$\textstyle \bE[g(S_t)] = \bE[g(m/F(t,B_t))F(t,B_t)] = \int g\left(\frac{m}{F\ast \gamma_{1-t}(y)}\right) F\ast \gamma_{1-t}(y)(\alpha\ast\gamma_t)(dy).$$
This gives us a quick access to computations involving the distribution of $S$ once we know $\alpha$ and $F$. These, in turn, can be computed efficiently using the fixed point scheme of \cite{CoHe21} or, equivalently, the measure preserving martingale Sinkhorn algorithm in \cite{JoLoOb23}; see \cite{AcMaPa23} for the proof of convergence. We note that Theorem \ref{thm:gBass structure irr} also allows for an autonomous description of the optimiser $S$ in $\mathbf{GP}_{\mu_0,\mu_1}$ as a solution to an SDE, see Proposition \ref{prop:gBass structure}. 
 
Having the connection between geometric and arithmetic problems at hand, the following result is immediate from \cite[Thm.~1.5]{BaScTs23} and \cite[Thm.~1.4]{BaBeScTs23}.
\begin{corollary}\label{cor:duality}
 Let $\mu_0,\mu_1\in \cP_{-1,1}(\bR_+)$ such that $\mu_0\preccurlyeq_{cx} \mu_1$. Then 
 \begin{align*} \mathbf{GP}_{\mu_0,\mu_1} &= \textstyle  \inf\left\{ \int \psi d\nu_1 - \int(\psi^\ast\ast\gamma_1)^\ast d\nu_0: \, \psi\mbox{ convex}\right\} \\
 &= \sup\left\{ MC(\nu_1,\alpha\ast\gamma_1) - MC(\nu_0,\alpha)  :\, \alpha \in \cP_2(\bR)\right \},
 \end{align*}
 where $\nu_i=Id_{\ddag}\mu_i$, $i=0,1$.
\end{corollary}

\begin{remark}\label{rk:mean1enough}
If $S$ is feasible for \eqref{eq:gp} for $(\mu_0,\mu_1)$, and $m=\int x\mu_0(dx)=\int x\mu_1(dx)$ then $\frac{1}{m}S$ is feasible for \eqref{eq:gp} for the measures $\tilde \mu_i = (x\to x/m)_\#\mu_i$, which have mean $1$. In particular, the value of the problem \eqref{eq:gp} is the same for $(\mu_0,\mu_1)$ and $(\tilde \mu_0,\tilde \mu_1)$ and the optimisers are constant multiples of each other. This explains why we take $\nu_i=Id_\ddag \mu_i=Id_\dag \tilde \mu_i$ in Theorem \ref{thm:g-a equiv}.
\end{remark}

{In section \ref{sec:PDE_dual} we provide an argument for Theorems \ref{thm:g-a equiv}-\ref{thm:gBass structure irr} in the irreducible case by means of duality and PDE techniques. In section \ref{sec:proof} we provide a probabilistic proof of Theorem \ref{thm:g-a equiv}, and in section \ref{sec:gBass structure} we prove Theorem \ref{thm:gBass structure irr} as a particular case of a more general statement wherein the irreducibility assumption is dropped.}

\section{Kantorovitch duality perspective on geometric Bass martingale}\label{sec:PDE_dual}
The geometric martingale Benamou--Brenier problem \eqref{eq:gmbb}  falls into the general class of optimal transportation under controlled stochastic dynamics. The duality for such problems is well understood, see \cite{TaTo13,GuoLoeper21}, and offers a rich source of insights into their structure. To wit, recently \cite{JoLoOb23} used this approach to present a PDE perspective on the Bass martingale and, in particular, offered an alternative justification for the duality result in \cite[Thm.~1.4]{BaBeScTs23} that we used above to obtain Corollary \ref{cor:duality}. 
We apply now an analogous approach to \eqref{eq:gmbb}. 
As explained in Remark \ref{rk:mean1enough}, without any loss of generality, we can assume that $\int x\mu_0(dx)=\int x\mu_1(dx)=1$. The dual to 
\eqref{eq:gmbb} is found by considering 
\begin{align}
\label{eq:dual}
\tag{G-Dual}
\mathbf{DualGmBB}_{\mu_0,\mu_1}&= \textstyle 
\sup_{u}\left\{
\int u(1,s)\mu_1(ds) - \int u(0,s)\mu_0(ds)\right\}
\end{align}
among the solutions $u$ of 
\beq\label{eq:main log}
\textstyle 
\dt u + \frac{\bar\sigma^2}{2}\frac{s^2\dss u}{1- s^2\dss u}=0,
\enq
which satisfy $s^2\dss u < 1$. We refer to \cite[Thm.~4.2]{TaTo13} for a statement allowing to derive the above, but note that our arguments remain formal. In particular, we assume the existence and uniqueness of the dual optimiser $u$, which may be hard to establish independently but which will follow, in the irreducible case, from our proofs in sections \ref{sec:proof} and \ref{sec:gBass structure}. Once we have the optimal $u$ in \eqref{eq:dual}, the optimal $\sigma$ in \eqref{eq:gmbb} is obtained directly via 
\beq\label{eq:sigma}
\textstyle 
\sigma = \frac{\bar\sigma}{1-s^2\dss u},
\enq
hence under $\bP$ we have $
\frac{d S_t}{S_t}=\frac{\bar\sigma}{1-s^2\dss u}d W^\bP_t=\sigma d W^\bP_t$.

A suboptimal but sufficient statement for our pedagogic purpose is the following:
\begin{proposition}
Let $\mu_0$ be a probability measure with a smooth density and $u$ a $C^4$ smooth classical solution to \eqref{eq:main log}. Consider  $\sigma$ in \eqref{eq:sigma} and $\mu_1\sim S_1$, where the process $(S_t)_{0\leq t\leq 1}$ has lognormal volatility $\sigma$ and $S_0\sim \mu_0$. Then  
\[
\mathbf{DualGmBB}_{\mu_0,\mu_1}=\mathbf{GmBB}_{\mu_0,\mu_1}
\]
holds, the l.h.s.\ is attained by $u$, the r.h.s.\ is attained by $S$, and $S_t=1/F(t,W^{\tilde \bP})$, for some $F, W^{\tilde \bP}$ obtained from $u$ as explained below.
\end{proposition}

At this point, it may seem that the optimal $\sigma$ depends on the reference level $\bar\sigma$, where we know from the equivalence between \eqref{eq:gmbb} and \eqref{eq:gp} that this does not happen. To understand this, we continue analogously to \cite[sec.~5]{Lo18}. We let
\[d\tP=S_1 d\bP \quad \text{on }\cF_1,\]  
and consider 
\[\textstyle v(t,s)=\left(-u(t,s)  - \ln(s){+\bar\sigma^2t/2}\right)/\bar\sigma.\] 
Direct verification shows that $v$ satisfies
$
\dt v -\frac{1}{2s^2\dss v} = 0
$
and $\sigma$ is derived from $v$ as 
$
\sigma=\frac{1}{s^2\dss v}
$,
and in particular is independent of $\bar\sigma$.
Differentiating \eqref{eq:main log}, It\^o's formula and 
 Girsanov's theorem yield the following direct observations.
\begin{lemma} We have
\begin{enumerate}
\item[(i)] $S_t\ds u(t,S_t), S_t\ds v(t,S_t)$ are local martingales under $\bP$;
\item[(ii)] $\ds u(t,S_t), \ds v(t,S_t)$ are local martingales under $\tP$; 
\item[(iii)] 
$d \ds v(t,S_t)= \frac{1}{S_t}dW^\tP_t$ for a $\tP$--Brownian motion $W^\tP$.
\end{enumerate}
\end{lemma}
The condition $s^2\dss u < 1$ implies that $v$ is strictly convex on its domain $(0,+\infty)$. Its Legendre transform, $v^*$, 
is therefore smooth, strictly convex and strictly increasing with 
$\dy v^*(t,\ds v(t,s))=s$, $s>0$, by the usual properties of Legendre transform. In particular, $v^*$ is invertible. Moreover $v^*$ satisfies
\[
\dt v^* + \dyy v^*/[2(\dy v^*)^2]=0,
\]
from the PDE satisfied by $v$. Then, letting 
\[
W_t= v^*(t,\ds v(t,S_t)),
\]
It\^o's formula gives
\[ 
d W_t = \dy v^*(t,\ds v(t,S_t))/S_t \cdot dW^\tP_t = d W^\tP_t.
\]
Therefore $W_t$ is a $\tP$--Brownian motion, and we can choose  $W^\tP_t\equiv W_t$. We let 
$
w=\left({v^*}\right)^{-1}, 
$
and note that $w$ inherits the regularity of $v$ and therefore of $u$. 
Then $w$ sends a $\tP$--Brownian motion onto a $\tP$--local martingale, 
 $ w(t,W_t)=\ds v(t,S_t)$. An application of It\^o's formula shows that $w$ must therefore satisfy the heat equation, i.e., 
$\dt w + \frac{1}{2}\dxx w = 0$. We finally have the following relationships:
\[
S_t\qquad\longleftrightarrow \qquad Y_t=\ds v(t,S_t) \qquad \longleftrightarrow \qquad W_t=v^*(t,Y_t),
\]
and, moreover,
$
S_t=\dy v^*(t,Y_t)=1/\dx w(t,W_t)
$, and we can now let $F=\dx w$.
Since $w$ solves the heat equation, so does $F=\dx w$.
This gives a fast numerical recipe for solving the HJB equation \eqref{eq:main log}.  Under $\tP$,  the process $(\frac{1}{S_t}:t\leq 1)$ is given by 
 $1/S_t = F\big(t, W_t\big)$ and hence is a $\tP$ Bass martingale, see Definition \ref{def:Bassmg} below.  
This  allows to efficiently simulate $(S_t:t\leq 1)$ and to price options, including path-dependent ones, via \eqref{eq:simulateGBass}. We explore this further in section \ref{sec:numerics} and link to recent works \cite{JoLoOb23, AcMaPa23, CoHe21} on numerics for \eqref{eq:ap}. We close this section with a summary of the main similarities and differences between the arithmetic and geometric Bass martingales. The former was denoted $(M_t)$ in \eqref{eq:ambb} but we write $(X_t)$ below keeping with notation of $x$ and $s$ for the state variables. 

%

\medskip 

\noindent
{\bf The Arithmetic Bass Martingale $(X_t:t\leq 1)$}
\\

\noindent\fbox{%
    \parbox{30pt}{%

\begin{align*}
X_t && \text{ Martingale under } \bP&& \textstyle  \vd X_t = \frac{1}{\dxx v}\vd W^\bP_t\\
Y_t=\dx v(t,X_t) && \text{ BM under } \bP&& \vd Y_t= \vd W^\bP_t\\
Z_t= v^*(t,Y_t) && \text{ Martingale under }\bP && \vd Z_t =  X_t\vd W^\bP_t \\
X_t =\dy v^*(t,Y_t) && \;v^*\text{ solves the heat equation.}
\end{align*}
}}
\\
\\
\\
{\bf The Geometric Bass Martingale $(S_t: t\leq 1)$}
\\

\noindent
\fbox{%
    \parbox{30pt}{%

\begin{align*}
S_t && \text{ Martingale under } \bP&& \textstyle  \vd S_t = \frac{1}{S_t\dss v}\vd W^\bP_t\\
Y_t=\ds v(t,S_t) && \text{ Martingale under } \tP&& \textstyle  \vd Y_t=\frac{1}{S_t} \vd W^\tP_t\\
W_t= v^*(t,Y_t) && \text{ BM under }\tP && \vd W_t = \vd W^\tP_t \\
\textstyle  \frac{1}{S_t} =  \dx w\big(t, W_t\big) && \;w=(v^*)^{-1}\text{ solves the heat equation.}
\end{align*}
}}

\section{A probabilistic proof of Theorem \ref{thm:g-a equiv} via a change of measure}
\label{sec:proof}


\begin{proof}[Proof of Theorem \ref{thm:g-a equiv}]
Let us denote $\cA^{GP}_{\mu_0,\mu_1}$ the $6$-tuples of 
$$\mathfrak S:= (\Omega, \cF, (\cF)_{t\geq 0}, \bP, (B_t)_{t\geq 0}, (S_t)_{t\in [0,1]}),$$ 
admissible for \eqref{eq:gp}. 
Likewise, the $6$-tuples $\mathfrak M:= (\Omega, \cF, (\cF_t)_{t\geq 0}, \tilde \bP, (W_t)_{t\geq 0}, (M_t)_{t\in [0,1]})$ admissible for \eqref{eq:ap} are denoted $\cA^{AP}_{\nu_0,\nu_1}$ . We denote $\mathcal L$ the operator that associates to such a 6-tuple the law of its 6-th element. 
 We say that $\alpha: \cA^{GP}_{\mu_0,\mu_1} \to \cA^{AP}_{\nu_0,\nu_1}$ is \emph{law-invariant} if $\mathcal L(\alpha(\mathfrak S^1)) = \mathcal L(\alpha(\mathfrak S^2))  $ whenever $\mathcal L(\mathfrak S^1) =\mathcal L(\mathfrak S^2) $, and in this case we simply write $\alpha(S)$, with similar notation/terminology for $\beta:  \cA^{AP}_{\nu_0,\nu_1} \to \cA^{GP}_{\mu_0,\mu_1}$.  
 
Denote $m=\int x \mu_0(dx)=\int x \mu_1(dx)$ and fix $\mathfrak S\in  \cA^{GP}_{\mu_0,\mu_1}$. 
Define $\tilde{\mathbb P}$, a probability measure on $\cF_1$, via $d\tilde{\mathbb P}:=\frac{S_1}{m}d\mathbb P$. We will use the notation $\tilde{\mathbb E}$ for expectation under this measure. {As $S$ is a non-negative martingale}, by Girsanov's theorem $\tilde{B_t}:=B_t-\int_0^t\sigma_s ds$ is a $\tilde{\mathbb P}$-Brownian motion. Note that $S_1>0$ $\bP$-a.s., and hence $\bP$ and $\tilde \bP$ are equivalent on $\cF_1$ and 
\begin{equation}\label{eq:Rmgcomp}
\textstyle 
M_1:=\frac{d\bP}{d\tilde \bP} \big|_{\cF_1}=\frac{m}{S_1}\quad \text{and}\quad M_t := \tilde{\bE}[M_1|\cF_t]=\frac{\mathbb E[\frac{S_1}{m}M_1|\mathcal F_t]}{\bE[\frac{S_1}{m}|\mathcal F_t]}=\frac{1}{\frac{S_t}{m}}=\frac{m}{S_t}, 
\end{equation}
for $t\in [0,1]$.
Hence 
\begin{equation}\label{eq:G-Apbequality}
\textstyle 
\mathbb E\left[\int_0^1\sigma_t dt \right] = \tilde{\mathbb E}\left[M_1\int_0^1\sigma_t dt\right]=\tilde{\mathbb E}\left[\int_0^1M_t\sigma_tdt\right]=\tilde{\mathbb E}\left[\int_0^1\Sigma_tdt\right], 
\end{equation}
where we defined $\Sigma_t:=M_t\sigma_t$. An application of the It\^o formula gives
\begin{align*}
    dM_t &= \textstyle  m \left[ \frac{-1}{S_t^2}S_t\sigma_tdB_t + \frac{S_t^2\sigma_t^2}{S_t^3}dt \right]  = -M_t\sigma_t dB_t + M_t\sigma^2_t dt \\
    &=-M_t\sigma_t[dB_t-\sigma_t dt]= \Sigma_t dW_t,
\end{align*}
where $W_t:=-\tilde{B_t}$ is likewise a $\tilde{\mathbb P}$-Brownian motion. 
Furthermore we observe that for any bounded, smooth test function $g$ we have
$$\textstyle \tilde{\mathbb E}[g(M_1)]=\mathbb E\left[\frac{g(M_1)}{M_1}\right] = \mathbb E\left[ g\left(\frac{m}{S_1}\right)\frac{S_1}{m}\right] = \int g\left(\frac{m}{y}\right)\frac{y}{m}\mu_1(dy)= \int gd\nu_1.$$
Similarly, using the martingale property of $S$ under $\mathbb P$, we have
$$\textstyle \tilde{\mathbb E}[g(M_0)]=\mathbb E  \left[\frac{g(M_0)}{ M_1}\right] = \mathbb E\left[ g \left(\frac{m}{S_0}\right)\frac{S_1}{m}\right ] =\mathbb E\left[ g \left(\frac{m}{S_0}\right )\frac{S_0}{m}\right] = \int gd\nu_0.$$
We conclude that $\mathfrak M:=(\Omega, \cF, (\cF_t)_{t\geq 0}, \tilde \bP, (W_t)_{t\geq 0}, (M_t)_{t\in [0,1]})\in \cA^{AP}_{\nu_0,\nu_1}$ and hence we have defined a map $\alpha: \cA^{GP}_{\mu_0,\mu_1} \to \cA^{AP}_{\nu_0,\nu_1}$ via $\alpha(\mathfrak S):=\mathfrak M$. 
In addition, \eqref{eq:G-Apbequality} holds and this implies that $\mathbf{GP}_{\mu_0,\mu_1}\leq \mathbf{AP}_{\nu_0,\nu_1}$. 
Furthermore, $\alpha$ is law-invariant, since for every suitable path-dependent test function $g$ we have
\begin{align*}
       \tilde \bE\Big[g\big(\{M_t: t\in [0,1]\}\big)\Big] &= \bE\Big[g\big(\{m/S_t:t\in[0,1]\}\big)\cdot \frac{S_1}{m}\Big].
\end{align*}

The reverse inequality is obtained in the same fashion by constructing $\beta : \cA^{AP}_{\nu_0,\nu_1}\to \cA^{GP}_{\mu_0,\mu_1}$. 
 Consider $\mathfrak M = (\Omega, \cF, (\cF_t)_{t\geq 0}, \tilde \bP, (W_t)_{t\geq 0}, (M_t)_{t\in [0,1]}) \in \cA^{AP}_{\nu_0,\nu_1}$, where $W$ is a $\tilde \bP$--Brownian motion,  $dM_t=\Sigma_t dW_t$ and $M_0\sim \nu_0$, $M_1\sim \nu_1$. Observe that $M_t>0$ a.s., $t\in [0,1]$ and let $\sigma_t:=\frac{\Sigma_t}{M_t}$. Define a new probability measure $\bP$ via $\frac{d\bP}{d\tilde\bP}\big|_{\cF_1}=M_1$, then $B_t := - W_t + \int_0^t \sigma_s ds $ is a $\bP$-Brownian motion. We let $S_t = \frac{m}{M_t}$ and by a direct computation, analogous to \eqref{eq:Rmgcomp}, we see that $S$ is a $\bP$-martingale and It\^o's formula gives $dS_t=S_t\sigma_t dB_t$. Finally, for a test function $g\geq 0$ we have
$$\textstyle \bE[g(S_1)]=\tilde \bE \left[g\left(\frac{m}{M_1}\right) M_1\right] = \int g\left(\frac{m}{y}\right)y\nu_1(dy) = \int g(y)\mu_1(dy)$$
using $\nu_1=Id_\ddag\mu_1$, and likewise $S_0\sim \mu_0$, so that 
$\beta(\mathfrak M):= (\Omega, \cF, (\cF_t)_{t\geq 0}, \bP, (B_t)_{t\geq 0},$ $ (S_t)_{t\in [0,1]}) \in \cA^{GP}_{\mu_0,\mu_1}$, as desired. 
 The equality of values \eqref{eq:G-Apbequality} still holds and we conclude that $\mathbf{GP}_{\mu_0,\mu_1}\geq \mathbf{AP}_{\nu_0,\nu_1}$, showing the two values are actually equal. 
 Further, for every suitable path-dependent test function $g$, we have
\begin{align*}
   \bE\Big[g\big(\{S_t: t\in [0,1]\}\big)\Big] &= \tilde \bE\Big[g\big(\{m/M_t:t\in[0,1]\}\big)\cdot M_1\Big],
\end{align*}
and hence $\beta$ is law-invariant. 
In addition, it follows directly by the above explicit constructions that $\beta(\alpha(\mathfrak S))=\mathfrak S$ and $\alpha(\beta(\mathfrak M))=\mathfrak M$.
 The equality of value functions \eqref{eq:G-Apbequality} implies that $\mathfrak S$ is an optimiser for \eqref{eq:gp} if and only if $\mathfrak M=\alpha(\mathfrak S)$ is an optimiser for \eqref{eq:ap}. Existence and uniqueness (in distribution) of the optimiser of \eqref{eq:ap} was established in \cite{BaBeHuKa20} and thus it immediately carries over to the geometric setting in \eqref{eq:gp}.
\end{proof}
 
\begin{remark}
From the proof of Theorem \ref{thm:g-a equiv} it is clear that several immediate generalisations are possible, linking arithmetic and geometric problems. Specifically, in analogy to \eqref{eq:G-Apbequality}, with $\mathfrak M=\alpha(\mathfrak S)$, we can write 
\begin{align*} \textstyle 
\mathbb E\left[\int_0^1c(t,S_t,\sigma_t^2) dt \right] = \tilde{\mathbb E}\left[\int_0^1M_tc(t,S_t,\sigma_t^2) dt\right]=\tilde{\mathbb E}\left[\int_0^1M_tc\left(t,\frac{m}{M_t},\frac{\Sigma_t^2}{M_t^2}\right) dt\right]. 
\end{align*}
Hence to $c:[0,1]\times \bR_+\times \overline{\bR_+} \to \bR$ we can associate $\tilde{c}(t,x,z):=xc(t,m/x,z/x^2)$, thus obtaining the equivalence of the 
geometric and arithmetic problem, respectively
\begin{align*} \textstyle 
\inf_{\substack{S_0\sim \mu_0, S_1\sim \mu_1\\ S_t=S_0+\int_0^t \sigma_s S_s dB_s}} 
\bE \left[\int_0^1c(t,S_t,\sigma_t^2)dt\right] &
\ \& \ \inf_{\substack{M_0\sim \nu_0, M_1\sim \nu_1\\ M_t=M_0+\int_0^t \Sigma_s dB_s}} 
\bE \left[\int_0^1\tilde c(t,M_t,\Sigma_t^2)dt\right],
\end{align*}
with $\nu_i=Id_{\ddag}\mu_i$, $i=0,1$. In this article we took $c(t,x,\sigma^2)=-\sigma$ as this is essentially the only case where the arithmetic version of the problem has an explicit solution.
\end{remark} 

\section{Structure of the geometric Bass martingale}
\label{sec:gBass structure}
We now turn to a detailed study and characterisation of the optimiser in \eqref{eq:gp}. Our proof in section \ref{sec:proof} shows that the optimiser $S$ to \eqref{eq:gp} is obtained by a change of measure procedure starting with the optimiser $M$ to \eqref{eq:ap}, which was dubbed \emph{stretched Brownian motion} and shown to be unique in \cite{BaBeHuKa20}. We recall a crucial concept towards understanding the structure of $M$.

\begin{definition}\label{def:Bassmg}
   A real-valued martingale $M$ is a Bass martingale, if there is an increasing function $F$ and a Brownian motion $B$ with a possibly non-trivial initial distribution such that 
    $$M_t=F(t,B_t),$$
    with $F(t,\cdot)=F*\gamma_{1-t}(\cdot)$. We refer to $F$ as the generating function and to $\alpha:=\text{Law}(B_0)$ as the Bass measure.
\end{definition}

Following \cite[Thm.~3.1]{BaBeHuKa20} or \cite[Thm.~1.3]{BaBeScTs23}, we know that the unique optimiser of \eqref{eq:ap} has a very specific structure, namely: conditionally on it starting in $I\in \mathcal I_{[\nu_0,\nu_1]}$, it is a Bass martingale. In other words, for any $I\in \mathcal I_{[\nu_0,\nu_1]}$, there exists $\alpha_I$ a probability measure and $F_{I}:\mathbb R\to \overline I$ increasing, such that, conditionally on $M_0\in I$, we have $M_t= F_{I}(t,W^I_t)$, with $F_{I}(t,x)=F_{I}\ast \gamma_{1-t}(x) $ and $W^I$ being a Brownian motion started according to $\alpha_I$. To justify this structure, note that in \eqref{eq:ap}, we have $\bE[\int_0^1 \Sigma_t dt] = \bE[M_1(B_1-B_0)]$ and by the usual transport results we expect $M_1$, conditionally on $B_0$, to be an increasing function of $B_1$.

To understand better the dynamics of the optimiser $S$ in \eqref{eq:gp}, we consider a representation for the optimiser $M$ in \eqref{eq:ap} where the Brownian motions $W^I$ are coupled together.
Let $(I_i)_{i\geq 1}$ be a numbering of intervals in $\mathcal I_{[\nu_0,\nu_1]}$ and $I_0:=\{z: U_{\nu_0}(z)=U_{\nu_1}(z)\}$, and write $F_{i}=F_{I_i}$, $\alpha_i=\alpha_{I_i}$. We consider a filtered probability space $(\Omega, \cF, (\cF_t)_{t\geq 0}, \tilde \bP)$ with a standard Brownian motion $\tilde B$ and $\cF_0$ rich enough to support 
an integer-valued random variable $\xi$ with $\tilde \bP(\xi=i)=\nu_0(I_i)$, $i\geq 0$, and $\zeta_0\sim {\nu_0}_{|I_0}/\nu_0(I_0)$, $\zeta_i\sim \alpha_i$, $i\geq 1$, all of these being $\cF_0$-measurable, {independent of each other}, and independent of the Brownian motion $\tilde B$. Let $W_0=\sum_{i\geq 1} \zeta_i \mathbf{1}_{\xi=i} $ and $W_t=W_0+\tilde B_t$. Let $M_1=\zeta_0\mathbf{1}_{\xi=0} + \sum_{i\geq 1} F_{i}(W_1)\mathbf{1}_{\xi=i}$ and $M_t=\tilde \bE[M_1|\cF_t]$. It follows that $M_t=\zeta_0\mathbf{1}_{\xi=0} + \sum_{i\geq 1} F_{i}(t,W_t)\mathbf{1}_{\xi=i}$.
We therefore have 
$$\textstyle  dM_t=M_t\left(\sum_{i\geq 1}\frac{\partial_x F_{i}(t,W_t)}{M_t}\mathbf{1}_{\xi=i}\right)dW_t = M_t \sigma_t dW_t,$$
where $\sigma_t := \sum_{i\geq 1}\partial_x \log(F_{i}(t,W_t))\mathbf{1}_{\xi=i} $. 
 Using the notation of the proof of Theorem \ref{thm:g-a equiv}, $\mathfrak S=\beta(\mathfrak M)$ is an optimiser in \eqref{eq:gp} and 
for any measurable test functional $g:C([0,1];\bR)\to\bR_+$ we have
$$
    \bE\Big[g\big(\{S_t: t\in [0,1]\}\big)\Big]=\tilde \bE\Big[g\big(\{m/M_t:t\in[0,1]\}\big)\cdot M_1\Big],$$
which fully characterises the distribution of $S$. 
It reduces to \eqref{eq:simulateGBass} in the irreducible case and thus establishes Theorem \ref{thm:gBass structure irr}. Note that the marginals of $S$ are recovered from the marginals of $M$ via $\bE[g(S_t)] = \tilde \bE[g(m/M_t)M_t]$, i.e., 
$$\textstyle S_t \sim \left(y \to \frac{y}{m}\right)_\dag \nu_t,
$$
where $\nu_t\sim M_t$. In the general, not necessarily irreducible setting, we note that for $J(I)\in \mathcal I_{[\mu_0,\mu_1]}$ associated to $I\in \mathcal I_{[\nu_0,\nu_1]}$ according to the bijection from Lemma \ref{lem:bijection}, we have $\{M_0\in I\}= \{S_0\in J(I)\}$. Importantly, $F_{i}(t,\cdot)$ is a smooth and strictly increasing function which admits an inverse, {for $t<1$}. On the set $\{S_0\in J(I_i)\}$, we have  $W_t=F_{i}^{-1}(t,m/S_t)$ and 
\begin{align*}
S_t &= \textstyle  S_0\exp\left\{-\int_0^t \partial_x \log F_i(s,W_s)dW_s + \frac{1}{2}\int_0^t  (\partial_x \log F_i(s,W_s))^2 ds\right\}\\
&=\textstyle  S_0 \exp \left\{\int_0^t \partial_x \log F_i(s,W_s)dB_s - \frac{1}{2}\int_0^t  (\partial_x \log F_i(s,W_s))^2 ds\right\}\\
&= \textstyle  S_0 \exp \left\{\int_0^t \frac{S_s}{m\partial_x F_i^{-1}(s,\frac{m}{S_s})}dB_s - \frac{1}{2}\int_0^t  \left(\frac{S_s}{m\partial_x F_i^{-1}(s,\frac{m}{S_s})}\right)^2 ds\right\}.
\end{align*}
We note that the last representation offers an intrinsic characterisation of the dynamics of $S$ under $\bP$ without the need to consider any dynamics under $\tilde \bP$. This is summarised in the following proposition.
\begin{proposition}\label{prop:gBass structure}
    Let $\mathfrak S\in \cA^{GP}_{\mu_0,\mu_1}$ be an optimiser in   \eqref{eq:gp}. Then, conditioned on $\{S_0\in J\}$, with $J\in \mathcal I_{[\mu_0,\mu_1]}$ corresponding to $I \in \mathcal I_{[\nu_0,\nu_1]}$, $S$ solves
$$ \textstyle  dS_t = S_t \frac{S_t/m}{\partial_x F_I^{-1}(t,\frac{m}{S_t})}dB_t,\quad 0<t<1.$$
\end{proposition}

\section{Relations between classes of Bass martingales}

As observed in \cite[Remark 1.9]{BaBeHuKa20}, the solution to \eqref{eq:ap} for lognormal marginals is given by the usual geometric Brownian motion which, trivially, is also the solution to \eqref{eq:gp} for lognormal marginals. So for lognormal marginals arithmetic Bass and geometric Bass martingales coincide and are equal to the geometric Brownian motion. This is clear from the following simple computation: let $\mu_0=\delta_{1}$ and $\mu_1$ be the distribution of $S_1$ for a  geometric Brownian motion $S$ solving $dS_t = \bar\sigma S_t dB_t$. Then, for a test function $g$ and $\nu_1=Id_\ddag \mu_1$ we have 
\begin{align*}
\textstyle \int g(y)\nu_1(dy) &= \textstyle \int g\left(\frac{1}{y}\right )y\mu_1(dy) = \int g(e^{-z})e^z e^{\frac{-(z+\bar\sigma^2/2)^2}{2\bar\sigma^2}}\frac{dz}{\bar\sigma\sqrt{2\pi}}\\
&=\textstyle  \int g(e^{-z})e^{\frac{-(z-\bar\sigma^2/2)^2}{2\bar\sigma^2}}\frac{dz}{\bar\sigma\sqrt{2\pi}} = \int g(e^{z})e^{\frac{-(z+\bar\sigma^2/2)^2}{2\bar\sigma^2}}\frac{dz}{\bar\sigma\sqrt{2\pi}}\\
&=\textstyle  \int g(y)\mu_1(dy),
\end{align*}
so $\mu_1 = Id_{\ddag}\mu_1$ is a fixed point for the $\ddag$ operator. We now show that in fact the geometric Brownian motion, with an arbitrary starting distribution, is the only process which is both an arithmetic and a geometric Bass martingale. 
\begin{proposition}
 Let $\mu_0,\mu_1\in \cP_2(\bR_+)$ in convex order $\mu_0\preccurlyeq_{cx} \mu_1$ and irreducible. Then the optimisers in \eqref{eq:ap} and \eqref{eq:gp} coincide if and only if for some $\sigma^2>0$:
 $$\log_\#{\mu_1} = \left(\log_\#{\mu_0}\right) \ast \cN(-\sigma^2/2,\sigma^2).$$
\end{proposition}
\begin{proof}
For sufficiency, let $\alpha = (y\to \frac{1}{\sigma}\log(y))_\# \mu_0$ and 
\begin{equation*}
S_t=\exp\left\{\sigma B_0 + \sigma (B_t-B_0)-\sigma^2t/2\right\},\quad t\in [0,1].
\end{equation*}
Then $S$ is a Bass martingale with $S_0\sim \mu_0$ and $S_1\sim \mu_1$, hence it attains $\mathbf{AP}_{\mu_0,\mu_1}$. Equally, since \eqref{eq:gmbb} and \eqref{eq:gp} are equivalent, the optimiser is the same for any choice of $\bar\sigma$ in \eqref{eq:gmbb}, but for $\bar\sigma=\sigma$, the process $S$ attains value zero, which is clearly the lower bound, and hence is the optimiser. 

For the converse implication, suppose $S$ is the optimiser in \eqref{eq:gp} and a Bass martingale. We have $S_t=H(t,B_t)$ for some function $H$ and a Brownian motion $B$. By the proof of Theorem \ref{thm:g-a equiv}, $S_t=\frac{m}{M_t}=\frac{m}{F(t,W_t)}$, for the $\tilde\bP$-Brownian motion $W$, since $M$ is a Bass martingale under $\tilde\bP$. It follows that, noting that both $H$ and $F$ are smooth for $0<t<1$ and strictly increasing in the spatial argument, 
$$\textstyle H^{-1}\left(t,\frac{m}{F(t,W_t)}\right) = B_t = -W_t + \int_0^t \partial_x \log F(s,W_s) ds$$
and hence, comparing the dynamics and equating the $dW_t$ terms,
$$\textstyle  \partial_x H^{-1}\left(t,\frac{m}{F(t,W_t)}\right) = -1\ d\tilde\bP-a.s.,\quad \Longrightarrow \frac{m}{F(t,x)} = H(t,-x+at),$$
for some constant $a$. Recall that both $H$ and $F$ solve the heat equation, so
\begin{align*}
 0 &=\textstyle  \partial_t F + \half \partial_{xx} F = -\frac{m}{H^2}\left(\partial_t H + a\partial_x H\right) + \frac{m}{2}\frac{-H^2\partial_{xx}H+ 2H(\partial_x H)^2}{H^4}\\
 & \textstyle = \frac{m}{H^2}\left(-\partial_t H - a\partial_x H -\half \partial_{xx} H + \frac{(\partial_x H)^2}{H}\right) =  \frac{m}{H^2}\left( - a\partial_x H + \frac{(\partial_x H)^2}{H}\right) 
\end{align*}
from which we deduce that  $\partial_x H  = aH$ and hence $H(t,x)=q(t)e^{ax}$, for some function $q$. Plugging into the heat equation we obtain $H(t,x)=e^{ax-a^2/2 t}$ as required. 
\end{proof}

To end this section, we discuss the relation between the martingale Benamou--Brenier problems and projections using the adapted Wasserstein distance. The latter, also known as the (bi)causal Wasserstein distance has been shown to be the natural analogue of the classical Wasserstein distance when measuring the distance between the distributions of two stochastic processes, see \cite{La18,AcBaZa20}. We focus here on the distance $\mathcal{AW}_2$ between martingale laws arising from measuring distances between paths $\omega,\omega'$ via $(\omega_0-\omega'_0)^2 + \langle \omega-\omega'\rangle_1$. Repeating the arguments used to show the equivalence between \eqref{eq:ambb} and \eqref{eq:ap}, and establishing suitable representation properties for bicausal couplings so that we can identify martingales with their distributions, \cite[Sec.~6]{BaBeHuKa20} show that \eqref{eq:ambb} is equivalent to $\mathcal{AW}_2^2$-projecting the Wiener measure on the set $\cM(\nu_0,\nu_1)$ of the distributions of continuous martingales with marginals $\nu_i$ at times $i=0,1$. At first, one might expect that the geometric Bass martingale in \eqref{eq:gmbb} arises as the $\mathcal{AW}_2^2$-projection of the geometric Brownian motion. This is not true. In fact, the resulting process is a continuous extension of the \emph{q}-Bass martingale recently introduced by \cite{Ts24}.

To see this, we consider a slightly more general setup. Let $M$ be an arithmetic Bass martingale with $M_1=F(B_1)$, where $(B_t)_{t\geq 0}$ is a standard Brownian motion on some stochastic basis. For simplicity, we assume $B_0$ is a constant and hence, also $M_0=F\ast \gamma_1(B_0)$ is a constant. Then, the problem of projecting $M$ onto $\cM(\mu_0,\mu_1)$ in the $\mathcal{AW}_2^2$-sense is equivalent to the problem 
$$\sup_{S\in \cM(\mu_0,\mu_1)}\bE[S_1F(B_1)],$$
where we continue to use the martingale and its distribution interchangeably for simplicity (and $B$ is supposed to be a Brownian motion in the same filtration where $S $ is a martingale). This amounts to a static weak optimal transport problem in the sense of \cite{GoRoSaTe14} which we can write as 
$$\sup_{\pi} \int MC(\pi^x,q)d\mu_0(x),$$
where now we optimize over one-period martingale couplings with the given marginals $\mu_0,\mu_1$ and where $q:=Law(M_1)$. To be precise, this static problem gives an upper bound to the above continuous problem. Its solution was recently studied in \cite{Ts24}, under the name of \emph{q}-Bass martingales. Assuming $q$ does not charge points, this problem admits a unique optimiser. While its full characterisation is, to the best of our knowledge, an open problem, in some situations \cite{Ts24} shows that the optimiser may be written as $S_1=G(\xi+F(B_1)),$ for an $\cF_0$-measurable random variable $\xi$ independent of $B$ and $G$ an increasing function. 
It then follows that $(G \ast q)_\# Law(\xi)= \mu_0$ and in our setting this means we can extend this one-period solution to a continuous martingale setting via
$$S_t=\bE[S_1|\cF_t] = \int G(\xi+ F(B_t+z))d\gamma_{1-t}(z) := G_t(\xi,B_t). $$
Note that $S$ has an absolutely continuous quadratic variation, and the above construction provides a bicausal coupling between $S$ and $M$ (recall that $M_t = F\ast \gamma_{1-t}(B_t)$), so the martingale $S$ saturates the upper bound, and is the optimiser to our projection problem. In particular, when $M=B$ we recover that $S$ is the \emph{arithmetic} Bass martingale. However, when $M$ is the geometric Brownian motion, or some other Bass martingale, the resulting projection appears to be a (continuous) \emph{q}-Bass martingale. We believe this provides a natural motivation to study these processes further and to understand better the $\mathcal{AW}_2^2$-projection relations between different types of Bass martingales. We leave this topic for future research.

\section{ Martingale  Sinkhorn systems}
\label{sec:numerics}
We work in the setting of Theorem \ref{thm:g-a equiv} and further w.l.o.g.\ assume
\beq\label{eq:defS0} \textstyle 
1=\int sd\mu_0(s)= \int sd\mu_1(s).
\enq 
For simplicity, we suppose $\mu_0,\mu_1$ are irreducible (otherwise our analysis has to be repeated for each irreducible component), and admit densities. For simplicity of notation, we identify the measure with its density. 

The unique Bass martingale which solves \eqref{eq:ambb} is characterised by $\alpha, F$ s.t.
\\
\noindent\fbox{%
    \parbox{410pt}{%
\beq\label{eq:Schro}
\nu_0 &=& (\gamma_1*F)_\# \alpha,\nonumber\\
\nu_1 &=& F_{\#}(\gamma_1*\alpha),\label{eq:MMPS}\\
\nu_t &=& (\gamma_{1-t}*F)_\#(\gamma_t*\alpha).\nonumber
\enq
}}
\\
The above system has been introduced in \cite{JoLoOb23} and referred to as the martingale Sinkhorn algorithm thanks to structural parallels with the classical Sinkhorn algorithm. It is equivalent to the fixed point characterisation in \cite{CoHe21}. Either formulation readily yields an algorithm to compute $\alpha,F$, see \cite{AcMaPa23, JoLoOb23} for a proof of its convergence.

In order to solve \eqref{eq:gp}, we may compute $\nu_i=Id_\dag \mu_i$, $i=0,1$, and solve the system \eqref{eq:MMPS} for $\nu_0, \nu_1$. This will yield $F, \alpha$, and for any time $t\in [0,1]$ we have
\[ \textstyle 
s\mu_t(s)=\left(\frac{1}{\gamma_{1-t}*F}\right)_\#(\gamma_t*\alpha).
\]

\bibliographystyle{siam}
\bibliography{joint_biblio}

\begin{thebibliography}{10}

\bibitem{AcBaZa20}
{\sc B.~Acciaio, J.~Backhoff-Veraguas, and A.~Zalashko}, {\em Causal optimal
  transport and its links to enlargement of filtrations and continuous-time
  stochastic optimization}, Stoch.~Proc.~Appl., 130 (2020), pp.~2918--2953.

\bibitem{AcMaPa23}
{\sc B.~Acciaio, A.~Marini, and G.~Pammer}, {\em Calibration of the bass local
  volatility model}, arXiv2311.14567,  (2023).

\bibitem{BaBeHuKa20}
{\sc J.~Backhoff-Veraguas, M.~Beiglb\"{o}ck, M.~Huesmann, and S.~K\"{a}llblad},
  {\em Martingale {B}enamou-{B}renier: {A} probabilistic perspective}, Ann.
  Probab., 48 (2020), pp.~2258--2289.

\bibitem{BaBeScTs23}
{\sc J.~{Backhoff-Veraguas}, M.~{Beiglb\"ock}, W.~{Schachermayer}, and
  B.~{Tschiderer}}, {\em The structure of martingale {B}enamou--{B}renier in
  multiple dimensions}, arXiv:2306.11019,  (2023).

\bibitem{BaScTs23}
{\sc J.~{Backhoff-Veraguas}, W.~Schachermayer, and B.~Tschiderer}, {\em {T}he
  {B}ass functional of martingale transport}, arXiv:2309.11181,  (2023).

\bibitem{Ba83}
{\sc R.~Bass}, {\em Skorokhod embedding via stochastic integrals}, in
  S{\'e}minaire de {Probabilit{\'e}s} {XVII} 1981/82, J.~Az{\'e}ma and M.~Yor,
  eds., no.~986 in Lecture {Notes} in {Mathematics}, Springer, 1983,
  pp.~221--224.

\bibitem{PHB}
{\sc M.~Beiglb{\"o}ck, P.~Henry-Labord{\`e}re, and F.~Penkner}, {\em
  Model-independent bounds for option prices---a mass transport approach},
  Finance Stoch., 17 (2013), pp.~477--501.

\bibitem{BeJoMaPa21b}
{\sc M.~Beiglb{\"o}ck, B.~Jourdain, W.~Margheriti, and G.~Pammer}, {\em
  Monotonicity and stability of the weak martingale optimal transport problem},
  Ann.~Appl.~Probab.,  (2023, to appear).

\bibitem{BeJu16}
{\sc M.~Beiglb{\"o}ck and N.~Juillet}, {\em On a problem of optimal transport
  under marginal martingale constraints}, Ann. Probab., 44 (2016), pp.~42--106.

\bibitem{BPR:24}
{\sc M.~Beiglböck, G.~Pammer, and L.~Riess}, {\em Change of numeraire for weak
  martingale transport}, arXiv:2406.07523,  (2024).

\bibitem{BrLi78}
{\sc D.~T. Breeden and R.~H. Litzenberger}, {\em Prices of state-contingent
  claims implicit in option prices}, J.~Business, 51 (1978), pp.~621--51.

\bibitem{BrHoRo01a}
{\sc H.~Brown, D.~Hobson, and L.~Rogers}, {\em Robust hedging of barrier
  options}, Math. Finance, 11 (2001), pp.~285--314.

\bibitem{CaLaMa14}
{\sc L.~{Campi}, I.~{Laachir}, and C.~{Martini}}, {\em {Change of numeraire in
  the two-marginals martingale transport problem}}, Finance Stoch., 21 (2017),
  pp.~471--486.

\bibitem{CoHe21}
{\sc A.~Conze and P.~Henry-Labord\`ere}, {\em Bass {Construction} with
  {Multi}-{Marginals}: {Lightspeed} {Computation} in a {New} {Local}
  {Volatility} {Model}}, SSRN:3853085,  (2021).

\bibitem{CoOb11a}
{\sc A.~M.~G. Cox and J.~Ob{\l}{\'o}j}, {\em Robust pricing and hedging of
  double no-touch options}, Finance Stoch., 15 (2011), pp.~573--605.

\bibitem{GaLaTo}
{\sc A.~Galichon, P.~Henry-Labord{\`e}re, and N.~Touzi}, {\em A stochastic
  control approach to no-arbitrage bounds given marginals, with an application
  to lookback options}, Ann. Appl. Probab., 24 (2014), pp.~312--336.

\bibitem{GoRoSaTe14}
{\sc N.~Gozlan, C.~Roberto, P.-M. Samson, and P.~Tetali}, {\em Kantorovich
  duality for general transport costs and applications}, J.~Funct.~Anal., 273
  (2017), pp.~3327--3405.

\bibitem{GuoLoeper21}
{\sc I.~Guo and G.~Loeper}, {\em {Path dependent optimal transport and model
  calibration on exotic derivatives}}, Ann.~Appl.~Probab., 31 (2021), pp.~1232
  -- 1263.

\bibitem{guo2022joint}
{\sc I.~Guo, G.~Loeper, J.~Ob{\l}{\'o}j, and S.~Wang}, {\em Joint modeling and
  calibration of {SPX} and {VIX} by optimal transport}, SIAM J.~Financial
  Math., 13 (2022), pp.~1--31.

\bibitem{guo2022calibration}
{\sc I.~Guo, G.~Loeper, and S.~Wang}, {\em Calibration of local-stochastic
  volatility models by optimal transport}, Math.~Finance, 32 (2022),
  pp.~46--77.

\bibitem{Ho98a}
{\sc D.~Hobson}, {\em Robust hedging of the lookback option}, Finance Stoch., 2
  (1998), pp.~329--347.

\bibitem{JoLoOb23}
{\sc B.~Joseph, G.~Loeper, and J.~Ob\l\'oj}, {\em The measure preserving
  martingale sinkhorn algorithm}, arXiv:2310.13797,  (2023).

\bibitem{KaSh91}
{\sc I.~Karatzas and S.~Shreve}, {\em Brownian motion and stochastic calculus},
  vol.~113, Springer Science \& Business Media, 1991.

\bibitem{La18}
{\sc R.~{Lassalle}}, {\em {Causal transference plans and their
  Monge-Kantorovich problems}}, Stoch.~Anal.~Appl., 36 (2018), pp.~452--484.

\bibitem{Lo18}
{\sc G.~Loeper}, {\em Option pricing with linear market impact and nonlinear
  {B}lack-{S}choles equations}, Ann. Appl. Probab., 28 (2018), pp.~2664--2726.

\bibitem{Ob04}
{\sc J.~Ob{\l}{\'o}j}, {\em The {S}korokhod embedding problem and its
  offspring}, Probab. Surv., 1 (2004), pp.~321--390.

\bibitem{PaRoSc22}
{\sc G.~Pammer, B.~A. Robinson, and W.~Schachermayer}, {\em A regularized
  {K}ellerer theorem in arbitrary dimension}, arXiv:2210.13847,  (2022).

\bibitem{TaTo13}
{\sc X.~Tan and N.~Touzi}, {\em Optimal transportation under controlled
  stochastic dynamics}, Ann. Probab., 41 (2013), pp.~3201--3240.

\bibitem{Ts24}
{\sc B.~Tschiderer}, {\em $q$-{B}ass martingales}, arXiv:2402.05669,  (2024).

\end{thebibliography}


\end{document}